\documentclass[12pt]{amsart}
\usepackage{amsmath,amsfonts,amsthm,amscd,amssymb,mathrsfs,amssymb,bm,pb-diagram, here, extsizes, midpage}
\usepackage{setspace}
\usepackage[all,cmtip]{xy}
\usepackage{tikz-cd}

\usepackage{blkarray}
\usepackage{multirow}
\usepackage{mathtools}
\usepackage{graphicx}

\usepackage{mathrsfs, arydshln}
\usepackage[all]{xy}
\newtheorem{theorem}{Theorem}

\newtheorem{proposition}[theorem]{Proposition}
\newtheorem{lemma}[theorem]{Lemma}
\newtheorem{definition}[theorem]{Definition}

\newcommand{\DDD}{\Delta}

\newcommand{\Lmd}{\Lambda}

\newcommand{\PP}{\mathbb{P}}
\newcommand{\CC}{\mathbb{C}}
\newcommand{\BB}{\mathbb{B}}

\newcommand{\RR}{\mathbb{R}}
\newcommand{\ZZ}{\mathbb{Z}}

\newcommand{\SU}{{\rm{SU}}}

\newcommand{\Sing}{{\rm{Sing}\,}}

\newcommand{\qdr}{\PP^1\times\PP^1}
\newcommand{\upi}{^{(i)}}

\newcommand{\upk}{^{(k)}}

\newcommand{\upone}{^{(1)}}

\renewcommand{\tilde}{\widetilde}

\newcommand{\mf}{\mathfrak}

\newcommand{\ol}{\overline}
\newcommand{\lra}{\longrightarrow}
\newcommand{\lras}{\,\longrightarrow\,}

\newcommand{\set}{\,|\,}

\newcommand{\proofend}{\hfill$\square$}

\newcommand{\inv}{^{-1}}

\newcommand{\Pic}{{\rm{Pic}}}

\newcommand{\ms}{\mathscr}
\newcommand{\minus}{\backslash}
\newcommand{\ptl}{\partial}

\newcommand{\qandq}{\quad{\text{and}}\quad}

\newcommand{\us}{^{\sigma}}

\DeclareFontFamily{U}{mathx}{}
\DeclareFontShape{U}{mathx}{m}{n}{<-> mathx10}{}
\DeclareSymbolFont{mathx}{U}{mathx}{m}{n}
\DeclareMathAccent{\widehat}{0}{mathx}{"70}
\DeclareMathAccent{\widecheck}{0}{mathx}{"71}

\numberwithin{equation}{section}
\numberwithin{theorem}{section}

\begin{document}
\bibliographystyle{alpha} 
\title[]
{Gravitational instantons and hyperbolic space}
\author{Nobuhiro Honda}
\address{Department of Mathematics, Institute
of Science Tokyo
}
\email{honda@math.titech.ac.jp}

\thanks{The author was partially supported by JSPS KAKENHI Grant 22K03308.
}
\begin{abstract}
We construct a degeneration of Hitchin's non-standard
minitwistor spaces associated with toric gravitational
instantons of type A$_{\rm odd}$ to the minitwistor space of a simple
hyperbolic orbifold. We show that the corresponding families
of minitwistor lines also converge under this degeneration.
Consequently, the associated non-standard Einstein--Weyl
spaces converge to the hyperbolic orbifold.
\end{abstract}

\maketitle

\setcounter{tocdepth}{1}
\section{Introduction}

Minitwistor spaces are complex surfaces that encode
three-dimensional Einstein--Weyl (EW) geometry. Under the
minitwistor correspondence, points of an EW space are
represented by real rational curves, called minitwistor
lines, in the corresponding minitwistor space
\cite{Hi82,JT85,PT93}. This framework was extended to
three-dimensional Severi varieties of nodal rational curves
in \cite{HN11}.

In \cite{Hi25}, Hitchin constructed a minitwistor space from
a toric ALE gravitational instanton of type $A_{2n-1}$
\cite{GH78,Kro89} by quotienting its twistor space by the
complexified scalar $S^1$-action. The resulting space
$\ms T$ is naturally realized as a double cover of the cone
$C(\Lmd)$ over a rational normal curve
$\Lmd\subset\PP^n$, branched along a real hyperelliptic curve
$\Sigma$ of genus $g=n-1$. Its minitwistor lines were
described in \cite{H26}.

In this paper, we show that the associated EW spaces can be regarded as
deformations of simple hyperbolic orbifolds. Let $\BB$ be the
hyperbolic three-ball and let $\ZZ_n$ act on it by rotations
around a fixed geodesic, where $\ZZ_n\subset SO(3)$ is the
image of the asymptotic group
$\ZZ_{2n}\subset SU(2)$ of the ALE space. The minitwistor
space of $\BB/\ZZ_n$ is $Q/\ZZ_n$, where
$Q=\PP^1\times\PP^1$. Like $\ms T$, this quotient is a double
cover of $C(\Lmd)$, now branched along a singular curve
$\Sigma_0$, which is reducible when $n$ is even. Letting half
of the branch points of $\Sigma$ coalesce at $0\in\Lmd$ and
the other half at $\infty\in\Lmd$, followed by a suitable
rescaling, gives a degeneration of $\ms T$ to $Q/\ZZ_n$.
A related Lorentzian degeneration to the standard de Sitter
space was studied in \cite{HN22}.

The quotient description also explains the nodes of the
limiting minitwistor lines. Away from the rotation axis, such
a line is the image of a smooth $(1,1)$-curve
$\tilde C\subset Q$, and intersections among its distinct
$\ZZ_n$-translates descend to the $n-1$ ordinary nodes of its
image. These nodes are all non-real when $n$ is odd, while
exactly one is real when $n$ is even; the real node corresponds
to a reflexive geodesic in the orbifold. Points on the
rotation axis give multiple minitwistor lines. When $n$ is
even, there is also a twisted family of minitwistor lines arising from non-real
$(1,1)$-curves whose images in the quotient are real.

Our main result is the following.

\begin{theorem}
Under the above degeneration of $\ms T$ to $Q/\ZZ_n$, the
family of minitwistor lines on $\ms T$ converges to the family
of minitwistor lines on $Q/\ZZ_n$ corresponding to the
hyperbolic orbifold $\BB/\ZZ_n$.
\end{theorem}

The convergence is understood in the sense of projective
$1$-cycles. To prove it, we first consider the symmetric
minitwistor lines, which form a two-dimensional slice
parameterized by a quarter of the hyperelliptic curve $\Sigma$ \cite{H26}. If $h$ is
the corresponding hyperplane section, then
$$
h|_{\Sigma}=q+\ol q+2D,
$$
where $D$ is its tangency divisor. Thus
$[D]\in\Pic^g(\Sigma)$ is a square root of
$[h|_{\Sigma}-q-\ol q]$. The boundary condition in \cite{H26} selects a distinguished
square root for each $q$, and these square roots vary
continuously with $q$. Their locus is called the Seifert surface.
Using the relative Picard
space, we extend this surface across the degeneration. The
key point is that the squaring map on the relevant Picard
component of $\Sigma_0$ is an isomorphism when $n$ is odd and
a two-to-one unramified covering when $n$ is even. Together
with an explicit description of the limiting curves, this
proves convergence for the symmetric slice; the full result
then follows from the $S^1$-action.

Section~2 describes the minitwistor lines on $Q/\ZZ_n$.
Section~3 constructs the degeneration of the minitwistor
spaces. Section~4 studies the symmetric minitwistor lines on
the limiting space, and Section~5 proves the convergence
theorem.

\section{The hyperbolic orbifold and nodal minitwistor lines}
\label{s:bmt}
We consider the complex projective space $\PP^3$ equipped with homogeneous coordinates $(X_0,X_1,X_2,X_3)$ and the real structure
$\sigma:(X_0,X_1,X_2,X_3)\longmapsto (\ol X_0,\ol X_1,\ol X_2,\ol X_3)$.
Let $Q\subset\PP^3$ be a smooth quadric defined by the equation 
$X_0^2=X_1^2+X_2^2+X_3^2$.
This is real (i.e.,\,$\sigma$-invariant), and letting $x_i=X_i/X_0$ for $i>0$, the real locus $Q\us$ is the unit sphere $x_1^2+x_2^2+x_3^2=1, x_i\in\RR$, sitting in $\{X_0\neq 0\}=\CC^3$.
Any real hyperplane that does {\em not} intersect $Q\us$ is defined by the equation of the form 
\begin{align}\label{hp1}
c_1x_1+c_2x_2+c_3x_3=1,\quad c_1^2+c_2^2+c_3^2<1,\quad
c_1,c_2,c_3\in\RR.
\end{align}
Hence, the space of such real hyperplanes is identified with the open unit ball 
$$
\BB=\big\{(c_1,c_2,c_3)\in\RR^3\set c_1^2+c_2^2+c_3^2<1\big\} \subset\RR^3.
$$
Sections of the quadric $Q$ by such hyperplanes are always smooth rational curves and they constitute the minitwistor lines on $Q$ which, by the Hitchin correspondence between minitwistor manifolds and 3-dimensional EW spaces \cite{Hi82}, induce the (positive) definite EW structure on $\BB $ that is simply the standard hyperbolic structure.
In the reverse direction, the minitwistor space $Q$, or more precisely, the complement $Q\minus Q\us$, is the space of oriented geodesics on $\BB $; under an identification $Q$ with the product $\qdr$ and 
regarding its two factors as the ideal boundary of $\BB$, each point $(q_-,q_+)\in Q\minus Q\us$ determines an oriented geodesic which has the two points $q_-$ and $q_+$ as its initial point and end point.

We note that if we let $c_1^2+c_2^2+c_3^2=1$ in \eqref{hp1},
then the hyperplane will be tangent to $Q$ at the point $(c_1,c_2,c_3)\in Q\us$, and the hyperplane section will be reducible, consisting of two distinct lines through the tangent point.
These hyperplane sections can be regarded as ``twistor lines'' that precisely correspond to points of the conformal boundary $\ptl\BB\simeq S^2$.
We call these {\em boundary minitwistor lines}.
These will be important in this paper. 

Rotations around any geodesic on $\BB$ are isometries.
Taking the $x_1$-axis as a geodesic, we denote by $R_{\theta}:\BB \lras \BB$ the rotation around it by angle $\theta$ (in some fixed direction).
Now, for any integer $n>1$, consider the group of isometries of $\BB$ generated by $R_{\theta}$ with $\theta = \frac{2\pi}n$.
This is of order $n$ and generates an action of the group $\ZZ_n=\ZZ/n\ZZ$ on $\BB$.
The quotient space $\BB /\ZZ_n$ is a hyperbolic orbifold having orbifold singularities along the image of the $x_1$-axis, and for the boundary, we have
\[
\ptl(\BB/\ZZ_n)
\simeq(\ptl\BB)/\ZZ_n
=S^2/\ZZ_n,
\]
whose underlying topological space is $S^2$.
We denote by $\Phi:\BB\lras\BB/\ZZ_n$ the quotient map.
In the following, we use the letter $R$ to mean the generating rotation 
$R_{\frac{2\pi}n}$ of the group $\ZZ_n$. 

Any isometry of $\BB$ induces a holomorphic automorphism of $Q$. 
We use the same symbol $R_{\theta}$ and $R$ for the holomorphic automorphism induced by the rotations on $\BB$.
If $\ZZ_n$ also denotes the group generated by $R$ acting on $Q$, then the quotient surface $Q/\ZZ_n$ is the minitwistor space of the hyperbolic orbifold $\BB/\ZZ_n$.
Introducing new homogeneous coordinates $Y_i$ on the ambient space $\PP^3$ by 
\begin{align}\label{Y}
Y_0=X_0+X_1,\quad
Y_1=X_0-X_1,\quad
Y_2=X_2+iX_3,\quad
Y_3=X_2-iX_3,
\end{align}
the equation of $Q$ will be $Y_0Y_1=Y_2Y_3$.
In the affine coordinates 
\begin{align}\label{afc}
(\zeta,\eta) := 
\Big(\frac{Y_2}{Y_0}, \,\frac{Y_3}{Y_0}\Big)
\end{align}
on $Q\simeq\qdr$,
the real structure $\sigma$ and the automorphism $R_{\theta}$ on $Q$ are respectively written as
\begin{align}
\sigma&:(\zeta,\eta) \longmapsto (\ol \eta,\ol \zeta),\label{rs0}\\
R_{\theta}&:(\zeta,\eta)\longmapsto \big(e^{i\theta}\zeta, e^{-i\theta}\eta\big).\label{Rn}
\end{align}
In particular, the real sphere $Q\us$ is exactly the anti-holomorphic diagonal $\{\zeta=\ol \eta\}$.

We use the same letter $\Phi$ to mean the quotient map $Q\lras Q/\ZZ_n$ under the $\ZZ_n$-action.
This preserves the real structure \eqref{rs0} and we use the same letter $\sigma$ for the real structure induced on $Q/\ZZ_n$.
Using \eqref{Rn}, the $\ZZ_n$-action on $Q$ has exactly four fixed points 
\begin{align}\label{sing1}
(0,0),(\infty,\infty), (0,\infty) \qandq (\infty,0),
\end{align}
and is free away from these points.
The images of the former two points are $A_{n-1}$-singularities of $Q/\ZZ_n$, which belong to $Q\us$, while the images of the latter two points are $A_{n,1}$-singularities of $Q/\ZZ_n$, which are $\sigma$-conjugate to each other.

The minitwistor lines on $Q/\ZZ_n$ corresponding to points of
$\BB/\ZZ_n$ are precisely the images under $\Phi$ of the
standard minitwistor lines on $Q$, and hence are all real.
We first study this family, including its multiple members and
the singularities of its non-multiple members. We then show
that, when $n$ is even, $Q/\ZZ_n$ carries another real family
of minitwistor lines, obtained from non-real $(1,1)$-curves on
$Q$, which we call the twisted family.

We begin with the multiple members. The singular locus of
$\BB/\ZZ_n$ is the image of the $x_1$-axis. The minitwistor
lines on $Q$ corresponding to points on this axis are given by
\begin{align}\label{mult}
\zeta\eta+b=0,\qquad b>0.
\end{align}
These curves are $\ZZ_n$-invariant. Hence their images under
$\Phi$, regarded as cycles, are smooth rational curves with
multiplicity $n$. We call them {\em multiple} minitwistor lines.
They are all real, and their reduced curves have real points
if and only if $n$ is even.
Thus, in that case, multiple minitwistor lines have a real circle although they correspond to points of the EW space which is definite.
We note that the two limiting curves obtained by letting $b=0,\infty$ in \eqref{mult} are the boundary minitwistor lines that correspond to the two points $(\pm 1,0,0)$ which are the two intersection points of $\ptl\BB$ with the $x_1$-axis.

The minitwistor lines on $Q$ which correspond to points of $\BB$ not belonging to the $x_1$-axis are defined by the equation of the form
\begin{align}\label{ab1}
\zeta\eta+a\zeta+\ol a\eta+b=0,
\qquad
a\in\CC^*,\,b\in\RR, \quad |a|^2-b<0.
\end{align}
In fact, these curves are all obtained by intersecting the hyperplanes \eqref{hp1} with $Q$, through the coordinate changes \eqref{Y} and \eqref{afc}. 
Hence, filling the positive part of the $b$-axis by the set of curves \eqref{mult}, 
the domain 
$$
\big\{(a,b)\in \CC\times\RR\set |a|^2<b\big\}\subset\CC\times\RR\simeq\RR^3
$$ gives a paraboloid model of the hyperbolic ball $\BB$. 
In this model, the rotation $R_\theta$ in \eqref{Rn} acts by
\begin{align}\label{S1act}
(a,b)\longmapsto(e^{-i\theta}a,b).
\end{align}
This $S^1$-action descends to an $S^1/\ZZ_n\simeq S^1$-action on $\BB/\ZZ_n$. 

We now determine the singularities of the images of the
curves in \eqref{ab1}.

\begin{proposition}\label{p:mtl1}
Let $\tilde C\subset Q$ be a curve of the form \eqref{ab1},
and put $C:=\Phi(\tilde C)\subset Q/\ZZ_n$. Then the
restriction
$
\Phi|_{\tilde C}:\tilde C\longrightarrow C
$
is the normalization map. Moreover, $C$ is a rational curve
having exactly $(n-1)$ ordinary nodes as its only
singularities, all of which lie away from the four singular
points of $Q/\ZZ_n$. If $n$ is odd, all these nodes are
non-real, whereas if $n$ is even, exactly one of them is real.
\end{proposition}

\proof
Write $\tilde C=\tilde C_0$, and put
$\tilde C_i:=R^i(\tilde C)$ for the $i$-th move of $\tilde C$ for $i<n$. If
$\omega=e^{2\pi i/n}$, then $\tilde C_i$ is given by
\eqref{ab1} with $a$ replaced by $\omega^{-i}a$. Since
$a\neq0$, the curves
$\tilde C_0,\tilde C_1,\ldots,\tilde C_{n-1}$ are mutually distinct.
Consequently,
\[
\Phi^*C=\sum_{i=0}^{n-1}\tilde C_i.
\]
Since distinct translates of $\tilde C$ intersect only in
finitely many points, the finite map
$\Phi|_{\tilde C}:\tilde C\to C$ is birational and hence is
the normalization map.

Each $\tilde C_i$ is a real $(1,1)$-curve without real
points, and hence it does not pass through $(0,0)$ or
$(\infty,\infty)$. Moreover, its equation, together with
$a\neq0$, shows directly that it passes through neither
$(0,\infty)$ nor $(\infty,0)$. Thus none of the curves
$\tilde C_i$ passes through the four fixed points
\eqref{sing1}.

Since the $\ZZ_n$-action is free away from these four points,
the singularities of $C$ arise from intersections among the
curves $\tilde C_i$, and the quotient does not change their
local analytic types. Any two distinct curves
$\tilde C_i$ and $\tilde C_j$ intersect transversally at two
points, which are exchanged by $\sigma$.
We claim that
\begin{align}\label{disjint}
(\tilde C_i\cap\tilde C_j)\cap(\tilde C_k\cap\tilde C_l)=\emptyset
\end{align}
whenever $i\neq j$, $k\neq l$, and $\{i,j\}\neq\{k,l\}$.

Indeed, if the intersection in \eqref{disjint} were non-empty, then,
using the real structure, we would have
$
\tilde C_i\cap\tilde C_j=\tilde C_k\cap\tilde C_l.
$
Let $p_\nu\in\BB$ be the point corresponding to $\tilde C_\nu$.
The intersection $\tilde C_i\cap\tilde C_j$ corresponds to the geodesic
through $p_i$ and $p_j$, and hence the above equality implies that
$p_i,p_j,p_k,p_l$ lie on the same geodesic.
Since at least three of the indices $i,j,k,l$ are distinct, this would
give three distinct points in the same $\ZZ_n$-orbit lying on one
geodesic.
But the $\ZZ_n$-orbit of a point away from the $x_1$-axis lies on a
Euclidean circle in a plane perpendicular to the axis, and no three
distinct points of such an orbit lie on a hyperbolic geodesic.
This proves \eqref{disjint}.

It follows that the set 
\[
\bigcup_{0\leq i<j<n}(\tilde C_i\cap\tilde C_j)
\]
consists of $n(n-1)$ points.
Since the $\ZZ_n$-action is free on this set, these points form exactly
$(n-1)$ $\ZZ_n$-orbits.
Thus $C$ has exactly $(n-1)$ ordinary nodes as its only singularities.

It remains to determine which of these nodes are real.
Let $z\in\tilde C_i\cap\tilde C_j$.
The node $\Phi(z)$ is real precisely when $z$ and $\ol z$ belong to the
same $\ZZ_n$-orbit.
So suppose that $\ol z=R^k(z)$ for some $0<k<n$.
After changing to the affine chart obtained by replacing $\zeta,\eta$
with $\zeta\inv,\eta\inv$ if necessary, write $z=(\zeta,\eta)$.
Then, with $\omega=e^{\frac{2\pi i}{n}}$, we have
\[
(\ol\eta,\ol\zeta)=(\omega^k\zeta,\omega^{-k}\eta),
\]
and hence $\omega^{2k}\zeta=\zeta$.
Since $\zeta=0$ would imply $\eta=0$, contradicting the absence of real
points on $\tilde C_i$, we have $\zeta\neq0$.
Thus $\omega^{2k}=1$, so $n$ is even and $k=n/2$.

In this case,
$R^{\frac n2}(z)=\ol z$ belongs both to
$\tilde C_i\cap\tilde C_j$ and to
$\tilde C_{i+\frac n2}\cap\tilde C_{j+\frac n2}$.
By \eqref{disjint}, this implies
$
\{i,j\}=\{i+(n/2),j+(n/2)\},
$
and hence $|i-j|=\frac n2$ modulo $n$.
Conversely, if $j=i+\frac n2$, then $R^{\frac n2}$ preserves the two-point set
$\tilde C_i\cap\tilde C_j$ and has no fixed point on it, so it exchanges
the two points.
Thus $R^{\frac n2}(z)=\ol z$, and the corresponding node is real.
As $i$ varies, these intersection points form a single $\ZZ_n$-orbit.
Therefore, all nodes are non-real when $n$ is odd, whereas exactly one
of them is real when $n$ is even.
\proofend

\medskip
Thus, the structure of all twistor lines on $Q/\ZZ_n$ which correspond to points of $\BB/\ZZ_n$ is now well understood. We next see that if {\em $n$ is even}, then $Q/\ZZ_n$ has another family of real minitwistor lines which are the images of {\em non-real} $(1,1)$-curves on $Q$ but whose EW space is again the hyperbolic orbifold $\BB/\ZZ_n$. 

\begin{proposition}\label{p:tw0}
Assume $n$ is even and 
consider the family of curves in $Q/\ZZ_n$, whose members consist of
the images of all $(1,1)$-curves on $Q$ defined by the equation of the form
\begin{align}\label{ab2}
\zeta\eta+a\zeta-\ol a\eta+ b=0,
\qquad
a\in\CC,\,b\in\RR,\quad |a|^2+b< 0.
\end{align}
Each member of this family on $Q/\ZZ_n$ is real, and these curves constitute a family
of minitwistor lines in the sense that they determine an EW
orbifold isomorphic to the hyperbolic orbifold $\BB/\ZZ_n$.
\end{proposition}

\proof
Consider the automorphism
$\Xi:Q\longrightarrow Q$ defined by
$\Xi(\zeta,\eta)=(\zeta,-\eta)$.
It commutes with the $\ZZ_n$-action and, since
$\sigma\circ\Xi=R^{\frac n2}\circ\Xi\circ\sigma$,
it induces a real automorphism of $Q/\ZZ_n$.
Moreover, $\Xi$ maps a curve of the form \eqref{ab1} with
coefficients $(a,b)$ to a curve of the form \eqref{ab2}.
Indeed, writing its coefficients as
$(a',b')=(-a,-b)$, the condition $|a|^2-b<0$ becomes
$|a'|^2+b'<0$, which is precisely the condition in
\eqref{ab2}.
\proofend

%

\begin{definition}\label{d:ord}
{\em
We call a non-multiple minitwistor line on $Q/\ZZ_n$
{\em ordinary} if it is the image of a curve of the form
\eqref{ab1}. When $n$ is even, we call it {\em twisted} if it is
the image of a curve of the form \eqref{ab2}.}
\end{definition}

We will need to take the twisted family into account when we prove the main result of this paper when $n$ is even.

\begin{proposition}\label{p:mtl2}
Let $C\subset Q/\ZZ_n$ be an ordinary or twisted minitwistor line. 
Then $C$ is a rational curve having exactly $(n-1)$ ordinary nodes as its only singularities, and they are away from the four singular points of $Q/\ZZ_n$.
If $n$ is odd, then all of them are non-real, while if $n$ is even, exactly one of them is real.
\end{proposition}

\proof
If $C$ is ordinary, this is exactly Proposition \ref{p:mtl1}.
If $C$ is twisted, then the image $\Xi(C)$ is ordinary, so it satisfies the required property. Since $\Xi$ is a real automorphism, the same property holds for $C$. \proofend

\section{A degeneration of minitwistor spaces}\label{s:degn}
In this section, we first realize the minitwistor space $Q/\ZZ_n$ of the hyperbolic orbifold $\BB/\ZZ_n$ as a branched double cover of a cone over a rational normal curve in $\PP^n$.
For this, we consider the linear system $|\ms O(n,n)|$ of bidegree $(n,n)$ on $Q\simeq\qdr$.
The $\ZZ_n$-action on $Q$ generated by the rotation $R$ of order $n$ naturally lifts to this linear system, 
and using the affine coordinates $\zeta,\eta$, the subsystem of $|\ms O(n,n)|$ consisting of $\ZZ_n$-invariant elements is generated by the following $(n+3)$ monomials:
\begin{align}\label{q1}
(\zeta\eta)^k, \,\,0\le k\le n, \qandq \zeta^n,\eta^n.
\end{align}
This subsystem induces a holomorphic map from $Q$ to $\PP^{n+2}$, which is $\ZZ_n$-invariant.
Hence, it descends to a map from the quotient $Q/\ZZ_n$ to $\PP^{n+2}$, which gives a projective embedding $Q/\ZZ_n\subset\PP^{n+2}$.
Thus, the quotient map $\Phi:Q\lras Q/\ZZ_n$ is realized by this $\ZZ_n$-invariant subsystem of $|\ms O(n,n)|$.

Next, the above $\ZZ_n$-action and the switching involution $\iota:(\zeta,\eta)\longmapsto
(\eta,\zeta)$ generate the dihedral group $D_n$ of order $2n$.
The subsystem of $|\ms O(n,n)|$ consisting of $D_n$-invariant elements is generated by the following $(n+2)$ polynomials:
\begin{align}\label{q2}
(\zeta\eta)^k, \,\,0\le k\le n, \qandq \zeta^n+\eta^n.
\end{align}
This also induces a $D_n$-invariant holomorphic map, from $Q$ to $\PP^{n+1}$ this time.
If $\Lmd\subset\PP^n$ denotes a rational normal curve,
then the image of this map is the cone $C(\Lmd)\subset\PP^{n+1}$ over $\Lmd$.

Comparing the sets of generators \eqref{q1} and \eqref{q2}, we obtain a commutative diagram of meromorphic maps
\begin{equation}\label{d1}
\begin{tikzcd}[column sep=large, row sep=large]
Q/\ZZ_n \arrow[r, hookrightarrow] \arrow[d, "\Pi_0"']
  & \PP^{n+2} \arrow[d] \\
C(\Lmd) \arrow[r, hookrightarrow]
  & \PP^{n+1},
\end{tikzcd}
\end{equation}
where the right vertical map is the linear projection from the point $(0:\dots:0:1:-1)\in\PP^{n+2}$, which is away from the image of $Q/\ZZ_n$, and $\Pi_0$ is its restriction to the image of $Q/\ZZ_n$ into $\PP^{n+2}$.
We may use 
\begin{align}\label{zv}
z:=\zeta\eta \qandq v:=\zeta^n+\eta^n
\end{align}
as affine coordinates on the cone, and 
from these, we obtain the equation $\zeta^{2n}-v\zeta^n+z^n=0$.
Taking the discriminant, we obtain:
\begin{proposition}\label{p:br1}
In the above coordinates $(z,v)$ on the cone $C(\Lmd)$, the map $\Pi_0:Q/\ZZ_n\lras C(\Lmd)$ is a double covering with branch given by the following double cover of $\Lmd$:
\begin{align}\label{discr}
\Sigma_0:=\big\{v^2=4z^n\big\}.
\end{align}
This curve has exactly two singularities, which are $A_{n-1}$-singularities lying over $z=0,\infty\in\Lmd$.
If $n$ is odd, then this is an irreducible rational curve, while if $n$ is even, then this consists of two smooth rational curves touching at the two singularities.
\end{proposition}

\proof
The curve \eqref{discr} has an $A_{n-1}$-singularity at the origin.
The other singularity can be seen by rewriting \eqref{discr} using a complementary affine chart $\tilde z=z\inv$ and $\tilde v = z^{-n}v$ on the cone.
The remaining properties are immediate.\proofend

\medskip
The real structure on $C(\Lmd)$ is naturally induced from that on $Q/\ZZ_n$ , and the two $A_{n-1}$-singularities of the branch curve $\Sigma_0$ are real. 
The vertex of $C(\Lmd)$ does not belong to $\Sigma_0$.
The inverse image of the vertex under $\Pi_0$ consists of the two
$A_{n,1}$-singularities of $Q/\mathbb Z_n$, and they are exchanged by $\sigma$.

Next, we recall from \cite{Hi25} and \cite{H26} the construction of a compact minitwistor space that arises from a toric ALE gravitational instanton of type $A_{2n-1}$. 
The 3-dimensional twistor space of such an instanton of type $A_{2n-1}$ is obtained from an algebraic variety 
\begin{align}\label{tw1}
xy = (z-a_1u)(z-a_2u)\dots (z-a_{2n}u),
\quad a_1<a_2<\dots<a_{2n}.
\end{align}
More precisely, this is considered as an equation in the total space of the rank-3 vector bundle
$\ms O(n)\oplus\ms O(n)\oplus\ms O(2)\lras\PP^1$, where $(x,y,z)$ are fiber coordinates of this bundle, $u$ is an affine coordinate on $\PP^1$, and $a_1,a_2,\dots,a_{2n}$ are real numbers which are determined by the monopole points in $\RR^3$ lying on a straight line.
By taking an appropriate resolution of the singularities of the variety \eqref{tw1}, 
one obtains the twistor space of the toric gravitational instanton of type $A_{2n-1}$ \cite{Hi78}.

The projective model \eqref{tw1} of the twistor space is invariant under the torus action 
\begin{align}\label{act1}
(x,y,z,u)\stackrel{(s,t)}\longmapsto (s^ntx,s^nt\inv y, sz,su),\quad(s,t)\in T^2=U(1)\times U(1).
\end{align}
The $S^1$-subgroup $\{s=1\}\subset T^2$ preserves each fiber of the projection to $\PP^1$, and this corresponds to the $S^1$-action that preserves all complex structures of the hyperK\"ahler structure.
On the other hand, the $S^1$-action of the subgroup $\{t=1\}\subset T^2$ is induced from the scalar multiplication on $\CC^2$ from which the $A_{2n-1}$-singularity arises as the quotient by a cyclic group of order $2n$ in $\SU(2)$.
We call this the {\em scalar $S^1$-action}.
At the ALE end, the scalar circle action induces the
Hopf fibration
\[
S^3/\ZZ_{2n}\longrightarrow S^2/\ZZ_n,
\]
where $\ZZ_n$ is the image of $\ZZ_{2n}\subset SU(2)$ in
$SO(3)$. Thus, after completing the three-dimensional
quotient at the ALE end, the added orbifold point has link
$S^2/\ZZ_n$, consistently with the local orbifold structure
of $\BB/\ZZ_n$.

From Jones-Tod \cite{JT85}, in suitable situations, a minitwistor space can be obtained from a 3-dimensional twistor space by taking the quotient by a $\mathbb C$- or $\mathbb C^*$-action. Such an action arises as the complexification of an $\mathbb R$- or $S^1$-action naturally lifted from an action preserving the (anti-)self-dual conformal structure.

Hitchin obtained a minitwistor space from the scalar $S^1$-action on (the resolution of) the twistor space \eqref{tw1}. 
More precisely, the fiber over the point $u=1\in\PP^1$, which is 
\begin{align}\label{tw2}
xy = (z-a_1)(z-a_2)\cdots (z-a_{2n}),
\end{align}
can be viewed as the orbit space of the complexified scalar $\CC^*$-action.
This surface admits a compactification in the $\PP^2$-bundle 
$\PP(\ms O(n)\oplus\ms O(n)\oplus\ms O)\lras\PP^1$
by still thinking $x$ and $y$ as fiber coordinates on the line bundle $\ms O(n)$, while $z$ is now regarded as an affine coordinate on the base $\PP^1$. 
We are still assuming $a_1<a_2<\dots<a_{2n}$.
The compactification is smooth.
As in \cite{H26}, we denote by $\tilde{\ms T}$ this compact complex surface. 
Although the fiber over $u=1$ is not real under the canonical real structure on the twistor space, 
$\tilde{\ms T}$ has a natural real structure, by taking the composition with the identification of the fibers over the two points $u=\pm1$ that is realized by the orbits of the scalar $\CC^*$-action.
In the above coordinates, it is explicitly given by 
\begin{align}\label{rs1}
(x,y,z)\longmapsto \big((-1)^n\ol y,(-1)^n\ol x,\ol z\big).
\end{align}
Further, the subgroup $\{s=1\}\subset T^2$ in \eqref{act1}
induces an $S^1$-action on $\tilde{\ms T}$.

In the following, to avoid the appearance of $(-1)^n$, we redefine $y$ to be $(-1)^ny$. Then the equation \eqref{tw2} becomes
\begin{align}\label{tw3}
xy = (-1)^n(z-a_1)(z-a_2)\cdots (z-a_{2n}),
\end{align}
while the real structure \eqref{rs1} is changed to a simple form as
\begin{align}\label{rs2}
(x,y,z)\longmapsto \big(\ol y,\ol x,\ol z\big).
\end{align}
The induced $S^1$-action is written as $(x,y,z)\longmapsto (tx,t\inv y,z)$, $t\in S^1$.

Through the compactification from the affine surface \eqref{tw2} or \eqref{tw3} to $\tilde{\ms T}$, three $\PP^1$ are attached. One of them is a (smooth) fiber over the point $z=\infty$, and the remaining two are sections of the projection to $\PP^1$ whose coordinate is $z$. 
The latter two are $(-n)$-curves in $\tilde{\ms T}$ and therefore each of them can be contracted to an $A_{n-1}$-singularity. 
As in \cite{H26} we denote by $\ms T$ the surface obtained by the contractions, and use this singular surface as a minitwistor space rather than $\tilde{\ms T}$.
The two $A_{n,1}$-singularities are mutually exchanged by $\sigma$.
For our purposes, it is useful that $\ms T$ admits a natural
realization as a double cover of the cone $C(\Lmd)$, as we
now explain.

Put $v=x+y$ and $w=x-y$. Then $(z,v)$ give coordinates on
the affine part of the cone $C(\Lmd)$, in agreement with the
notation \eqref{zv}. In these coordinates, equation
\eqref{tw3} becomes
\begin{align}\label{tw4}
w^2 = v^2 - 4(-1)^n(z-a_1)(z-a_2)\cdots (z-a_{2n}),
\end{align}
while the real structure \eqref{rs2} becomes $(v,w,z)\longmapsto (\ol v,-\ol w,\ol z)$.
The equation 
\begin{align}\label{he}
v^2=4(-1)^n(z-a_1)(z-a_2)\cdots (z-a_{2n})
\end{align}
 makes sense as the one defined in the total space of the line bundle $\ms O(n)\lras\PP^1$, namely the cone over a rational normal curve $\Lmd$ in $\PP^n$, and it defines a hyperelliptic curve of genus $(n-1)$.
Therefore, from \eqref{tw4}, the minitwistor space $\ms T$ has a structure of a double covering
\begin{align}\label{dc1}
\Pi:\ms T\lras C(\Lmd)
\end{align}
of the same cone $C(\Lmd)$ as $Q/\ZZ_n$, branched along the hyperelliptic curve \eqref{he}. 
We denote by $\Sigma$ this hyperelliptic curve.
This is of degree $2n$ in $\PP^{n+1}$ and invariant under the real structure $\sigma:(v,z)\longmapsto (\ol v,\ol z)$ induced from the above one.
The curve $\Lmd$, the cone $C(\Lmd)$ and the minitwistor space $\ms T$ are realized in $\PP^n$, $\PP^{n+1}$ and $\PP^{n+2}$ respectively.

The inverse image of the real circle $\Lmd\us=\{z\in\RR\}\cup\{\infty\}$ under the double covering map $\Sigma\lras\Lmd$ consists of $2n$ smooth circles, according as the sign of the right-hand side of \eqref{he}. Half of these circles consist of real points, and the remaining half consist of points whose $v$-coordinate is pure imaginary.
As in \cite{H26}, we call these {\em real circles} and {\em pure imaginary circles} respectively.
These are arranged alternately, and the intersections of two adjacent circles are exactly the ramification points of the double cover $\Sigma\lras\Lmd$.

The complement of the union of all real and pure imaginary circles in $\Sigma$ consists of four domains, and each of them is a fundamental domain of the action on $\Sigma$ by the group generated by the real structure and the hyperelliptic involution.
As in \cite[Section 6]{H26}, we call the closure of any one of these domains a {\em quarter} of the hyperelliptic curve $\Sigma$, and denote $\Sigma''$ for it.
This is a manifold with corners and is naturally identified with a {\em slice} of the $S^1$-action of the EW orbifold associated to the toric $A_{2n-1}$ gravitational instanton.
We use this for proving the convergence of the minitwistor lines.

The minitwistor space $\ms T$ depends on the real $2n$ parameters $a_1,\dots,a_{2n}$, arranged on $\RR$ in this order.
We next show that $\ms T$ converges to $Q/\ZZ_n$ in a certain limit of these $2n$ points:

\begin{proposition}
In the limit
\begin{align}\label{lim1}
a_1,\dots,a_n\lras 0\qandq
a_{n+1},\dots,a_{2n}\lras\infty,
\end{align}
the minitwistor space $\ms T$ arising as the scalar $S^1$-quotient of the
toric gravitational instanton of type $A_{2n-1}$ is deformed into
$Q/\ZZ_n$, the minitwistor space of the hyperbolic orbifold $\BB/\ZZ_n$.
\end{proposition}

\proof
In the limit $a_{n+1},\dots,a_{2n}\lras\infty$, we have
\[
\prod_{j=n+1}^{2n}(z-a_j)
=
(-1)^n A(1+o(1)),\qquad
A:=a_{n+1}\cdots a_{2n}>0.
\]
After replacing \(v\) by \(v/\sqrt A\) and letting $a_1,\dots,a_n\lras 0$, the equation \eqref{he}
therefore converges to
\[
v^2=4z^n.
\]
This is exactly the equation \eqref{discr} of the branch divisor $\Sigma_0$ of the double cover $\Pi_0:Q/\ZZ_n\lras C(\Lmd)$ given in Proposition \ref{p:br1}.
This implies that $Q/\ZZ_n$ is obtained as a deformation of $\ms T$ through the limit \eqref{lim1}.
\proofend

\medskip
This proposition, however, does not by itself show that the former EW space deforms to the latter hyperbolic orbifold. For this, one also needs to show that the minitwistor lines in the former minitwistor space converge to those in $Q/\ZZ_n$.
In the rest of this paper, we show that this is really the case.

\section{Symmetric minitwistor lines on the limiting space}

In \cite{H26}, to each point $q$ of a quarter $\Sigma''$ of
the hyperelliptic branch curve $\Sigma$, we associate a real
hyperplane $h\subset\PP^{n+1}$ whose only non-tangential
intersection points with $\Sigma$ are $q$ and $\ol q$.
(Recall that $\deg\Sigma=2n$ in $\PP^{n+1}$.)
The inverse image $\Pi\inv(h)$ of such $h$ is a minitwistor line in $\ms T$ which is invariant under the covering transformation of $\Pi$. 
Hence, we obtain a 2-dimensional family of minitwistor lines parameterized by $\Sigma''$.
The full 3-dimensional family of minitwistor lines on $\ms T$ is obtained by 
moving these symmetric ones under the $S^1$-action of $\{s=1\}\subset T^2$.
In this section, we shall see that an analogous situation occurs in the limit
$Q/\ZZ_n$, with respect to the double covering
$\Pi_0:Q/\ZZ_n\lras C(\Lmd)$.
The following definition also applies to multiple and
boundary minitwistor lines.

\begin{definition}\label{d:sym}
{\em
A minitwistor line in either $\ms T$ or $Q/\ZZ_n$ is called
{\em symmetric} if it is invariant under the covering
transformation of $\Pi$ or $\Pi_0$, respectively.
}
\end{definition}
The covering transformation of $\Pi_0$ is induced from the switching involution $\iota$ on $Q$.
In fact, as $R\circ\iota=\iota\circ R^{-1}$, $\iota$ maps $\ZZ_n$-orbits to $\ZZ_n$-orbits and hence it descends to an involution on the quotient $Q/\ZZ_n$. We use the same symbol for it.
Then comparing the generators given in \eqref{q1} and
\eqref{q2}, this induced involution is precisely the
covering transformation of $\Pi_0$.

We recall that the hyperbolic orbifold $\BB/\ZZ_n$ admits an effective $S^1$-action induced from that on $\BB$ as written in \eqref{Rn} or \eqref{S1act}. 

\begin{proposition}\label{p:sym}
(i)
Every $S^1$-orbit in the ordinary family contains a unique symmetric representative of the form
\begin{align}\label{ab3}
\zeta\eta+a(\zeta+\eta)+b=0,
\qquad
a,b\in\RR,\quad a> 0,\quad a^2- b<0.
\end{align}
(ii)
When $n$ is even, every $S^1$-orbit in the twisted family contains a unique symmetric representative of the form
\begin{align}\label{ab4}
\zeta\eta+a(\zeta-\eta)+b=0,
\qquad
a,b\in\RR,\quad a>0,\quad a^2+b<0.
\end{align}
\end{proposition}

\proof
In both the ordinary and twisted families, the $S^1$-action rotates the
coefficient $a$ while leaving $b$ unchanged; see \eqref{S1act}.
Since $a\neq0$ for a non-multiple minitwistor line, each $S^1$-orbit
contains a unique member for which $a$ is positive real.

For the ordinary family, this gives \eqref{ab3}, which is invariant
under $\iota$.
For the twisted family, this gives \eqref{ab4}, and
$\iota(\tilde C)=R^{n/2}(\tilde C)$.
Hence its image in $Q/\ZZ_n$ is invariant under the covering
transformation.
Thus the representatives in \eqref{ab3} and \eqref{ab4} are symmetric,
and their uniqueness follows from the uniqueness of the above choice
of the phase of $a$.
\proofend

%
%

\medskip
Thus, the orbit space of the $S^1$-action on the space of ordinary minitwistor lines in $Q/\ZZ_n$ is identified with 
 the domain
\begin{align}\label{dom1}
\ms D_{\rm ord}
:=
\big\{(a,b)\in\RR^2\mid a>0, \,a^2-b<0\big\},
\end{align}
and in the case $n$ is even, 
the orbit space of the $S^1$-action on the space of twisted minitwistor lines in $Q/\ZZ_n$ is identified with the domain
\begin{align}\label{dom2}
\ms D_{\rm tw}
:=
\big\{(a,b)\in\RR^2\mid a>0,\,a^2+b<0\big\}.
\end{align}
We will use these domains to prove the convergence of the minitwistor
lines in question.



Note that if $a=0$ in \eqref{ab3}, then the curve becomes a $\ZZ_n$-invariant minitwistor lines, and such lines are parameterized by $b\in\RR_{>0}$.
These are placed at a part of the boundary of the domain \eqref{dom1}.
The images of these minitwistor lines to the cone $C(\Lmd)$ are generating lines with multiplicity $n$, which are over the point $z=\zeta\eta<0$.
The inverse images under $\Pi_0$ of the generating lines over
$z=0$ and $z=\infty$ each split into two lines. The incidence
graph of these four lines is a square.

The images of symmetric minitwistor lines under the double covering $\Pi_0:Q/\ZZ_n\lras C(\Lmd)$ can be described as follows.

\begin{proposition}\label{p:br2}
If $C\subset Q/\ZZ_n$ is a symmetric minitwistor line which
is non-multiple and is not a boundary minitwistor line, then $\Pi_0$ maps $C$ two-to-one onto its image $\Pi_0(C)\subset C(\Lmd)$, and $\Pi_0(C)$ is a real hyperplane section of the cone $C(\Lmd)$ whose defining hyperplane is tangent to the branch curve
$\Sigma_0$ at exactly $(n-1)$ points.
Moreover, the defining hyperplane passes through neither the vertex of the cone nor the two singularities of the branch curve $\Sigma_0$.
\end{proposition}

\proof
Choose a lift $\tilde C\subset Q$ of $C$.
Since $C$ is not multiple, $\tilde C$ has trivial stabilizer under
$\ZZ_n$, and hence $\Phi|_{\tilde C}:\tilde C\lras C$ is birational.
Since the projective embedding $Q/\ZZ_n\lras\PP^{n+2}$ in
\eqref{d1} is induced from a subsystem of $|\ms O(n,n)|$ on $Q$,
we have
\[
\deg C=(n,n)\cdot(1,1)=2n.
\]
On the other hand, since $C$ is symmetric, the restriction
$\Pi_0|_C:C\lras\Pi_0(C)$ has degree two.
As $\Pi_0$ is induced by the linear projection in \eqref{d1}, it follows
that $\Pi_0(C)$ has degree $n$ in $\PP^{n+1}$.

Since $C$ is not multiple, the function $z=\zeta\eta$ is non-constant
on $\tilde C$.
Hence the projection of $\Pi_0(C)$ from the vertex of the cone onto
$\Lmd$ is surjective.
Thus the linear span of $\Pi_0(C)$ has dimension at least $n$.
Since an irreducible curve of degree $n$ spans a projective space of
dimension at most $n$, its span is a hyperplane
$h\subset\PP^{n+1}$.
As the cone $C(\Lmd)$ also has degree $n$, we obtain
$
\Pi_0(C)=h\cap C(\Lmd).
$
Since $C$ and $\Pi_0$ are real, the hyperplane $h$ is real.

Suppose that $C$ is ordinary.
Then Proposition \ref{p:mtl1} shows that $C$ has exactly $(n-1)$
ordinary nodes as its only singularities and avoids the four singular
points of $Q/\ZZ_n$.
If $C$ is twisted, then $n$ is even, and the automorphism $\Xi$
in the proof of Proposition \ref{p:tw0} maps an ordinary minitwistor line onto
$C$.
Since the induced automorphism of $Q/\ZZ_n$ preserves the four
singular points, the same conclusions hold for $C$.
The inverse image under $\Pi_0$ of the vertex of the cone consists of
two of these four singular points, while the inverse images of the two
singular points of $\Sigma_0$ are the other two.
Since $C$ is invariant under the covering transformation, we have
$
C=\Pi_0^{-1}\bigl(\Pi_0(C)\bigr).
$
It follows that $h$ passes through neither the vertex of the cone nor
the two singular points of $\Sigma_0$.

Hence all intersections of $h$ with $\Sigma_0$ take place at smooth
points of $\Sigma_0$.
Locally at such a point, a simple tangency of $h\cap C(\Lmd)$ with
$\Sigma_0$ gives an ordinary node of $C$, whereas a transverse
intersection gives a smooth point of $C$.
Since $C$ has exactly $(n-1)$ ordinary nodes and no other
singularities, $h$ is tangent to $\Sigma_0$ at exactly $(n-1)$ points.
\proofend

\medskip
The converse of this proposition holds as follows. This will be used at the final step in the proof of our main theorem.

\begin{proposition}\label{p:br3}
Let $h\subset\PP^{n+1}$ be a real hyperplane which passes
through neither the vertex of $C(\Lmd)$ nor the singularities of
$\Sigma_0$, and suppose that
\begin{align}\label{hSigma0}
h|_{\Sigma_0}=q+\ol q+2D,
\end{align}
where $q$ is a non-real point and $D$ is a reduced effective divisor of degree
$(n-1)$.
Then
$\Pi_0^{-1}\bigl(h\cap C(\Lmd)\bigr)$
is a symmetric minitwistor line in $Q/\ZZ_n$ which is either ordinary or twisted.
\end{proposition}

\proof
Put
$\Gamma:=h\cap C(\Lmd)$ and 
$C:=\Pi_0^{-1}(\Gamma)$. The latter is obviously symmetric.
Since $h$ does not pass through the vertex of the cone, $\Gamma$ is a smooth
rational curve.
From the assumption \eqref{hSigma0}, 
the curve $C$ has an ordinary node over each point of $D$, while its
normalization is a double cover of $\Gamma$ branched precisely at $q$ and
$\ol q$.
Hence $C$ is an irreducible real rational curve having exactly
$(n-1)$ ordinary nodes.

Put $X:=\Phi^{-1}(C)\subset Q.$
Since $h$ passes through neither the vertex of the cone nor the
singularities of $\Sigma_0$, the curve $C$ avoids the four singular
points of $Q/\ZZ_n$.
Hence the restriction
$
\Phi|_X:X\lras C
$
is an unramified covering of degree $n$.
In particular, $X$ is reduced and has exactly $n(n-1)$ ordinary nodes
and no other singularities.

On the other hand, $X$ is a divisor of bidegree $(n,n)$ on $Q$, and
therefore
$
p_a(X)=(n-1)^2.
$
Let $X_1,\ldots,X_r$ be the irreducible components of $X$, and let
$g_i$ be the genus of the normalization of $X_i$.
Since $X$ has $n(n-1)$ nodes, from the normalization exact sequence, we obtain the formula
\[
(n-1)^2
=
\sum_{i=1}^r g_i+n(n-1)-r+1.
\]
This gives
$
\sum_{i=1}^r g_i=r-n,
$
and hence $r\geq n$.
On the other hand, each $X_i$ is mapped surjectively onto $C$ by
$\Phi$, and the sum of the degrees of these maps is $n$.
Therefore $r\leq n$.
Consequently $r=n$,  $g_i=0$ for all $i$, and each restriction
$\Phi|_{X_i}:X_i\lras C$ has degree one.
The $\ZZ_n$-action permutes these $n$ components transitively.
Hence they all have the same bidegree.
Since their sum has bidegree $(n,n)$, each of them has bidegree
$(1,1)$.
In particular, every $X_i$ is a smooth rational curve.


Let $\tilde C$ be one of these $(1,1)$-components.
Since $C$ is real, we have
$\sigma(\tilde C)=R^k(\tilde C)$ for some $k$.
As the stabilizer of $\tilde C$ is trivial, applying $\sigma$
twice gives
$2k\equiv0\pmod n$.
Hence either $k=0$, or $n$ is even and $k=n/2$.
In the former case, comparison of the coefficients shows that
$\tilde C$ has an equation of the type appearing in
\eqref{ab1}, while in the latter case it has an equation of
the type appearing in \eqref{ab2}. Moreover, the symmetry of
$C$ gives $\iota(\tilde C)=R^l(\tilde C)$ for some $l$, and
the non-reality of $q$ and $\ol q$ then yields
$|a|^2-b<0$ in the former case and $|a|^2+b<0$ in the latter.
Since $\Phi^{-1}(C)$ has $n$ distinct components, $C$ is not
multiple. Thus $C$ is respectively an ordinary or a twisted
minitwistor line.
%
%
\proofend

\medskip
As in the setting of Propositions \ref{p:br2} and
\ref{p:br3}, let
$C=\Phi(\tilde C)\subset Q/\ZZ_n$ be a non-multiple
symmetric minitwistor line which is not a boundary
minitwistor line, and let $h\subset\PP^{n+1}$ be the
hyperplane such that $C=\Pi_0\inv(h)$.
Then, as in \eqref{hSigma0}, we have
\begin{align}\label{hSigma1}
h|_{\Sigma_0}=q+\ol q+2D,
\end{align}
where $q,\ol q$ are non-real points and $D$ is a real reduced effective
divisor of degree $g$.
We call $q,\ol q$, as points on $\Sigma_0$ or the corresponding points
on $Q/\ZZ_n$, the {\em non-tangential points}, since they are the only
points at which $h$ meets the branch curve $\Sigma_0$ non-tangentially.
We also call $D$ the {\em tangency divisor} of $h$.
Using the non-tangential points rather than the tangency points will be
the key to proving the convergence of the minitwistor lines.

From $z=\zeta\eta$ and $v=\zeta^n+\eta^n$ as in \eqref{zv},
the inverse image of the branch divisor $\Sigma_0=\{v^2=4z^n\}$ of the double covering $\Pi_0$ by the composition $\Pi_0\circ\Phi:Q\lras C(\Lmd)$ consists of the $n$ curves 
\begin{align}\label{br2}
\tilde\Sigma\upk:=\big\{(\zeta,\eta)\in Q\set \eta= \omega^k\zeta\big\},\quad 0\le k<n.
\end{align} 
All these are real.
The group $\ZZ_n$ acts on the set of these $n$ curves.
We have $R(\tilde\Sigma\upk) = \tilde\Sigma^{(k-2)}$, so
if $n$ is odd, then the action is free and transitive, and
$\Phi$ restricts to a homeomorphism from $\tilde \Sigma\upk$ to $\Sigma_0$.
This restriction is the normalization of $\Sigma_0$, which is irreducible.
If $n$ is even, the $\ZZ_n$-action has exactly two orbits, those
through $\tilde\Sigma^{(0)}$ and $\tilde\Sigma\upone$.
The isotropy subgroup of each $\tilde\Sigma\upk$ is
$\langle R^{n/2}\rangle$, and
$\Phi$ restricts to a double covering from $\tilde\Sigma\upk$
onto one of the two components of $\Sigma_0$.
It is branched at the two points whose images are the two singular
points of $\Sigma_0$.
The image component depends on the parity of $k$.

Next, we will identify the subset of the branch divisor $\Sigma_0$ formed by one of the non-tangential points as a symmetric minitwistor line $C\subset Q/\ZZ_n$ moves.
For this purpose, we first show the following

\begin{lemma}\label{l:q}
If a lift $\tilde C$ of a symmetric ordinary minitwistor line $C$ is defined 
by $\zeta\eta+a(\zeta+\eta)+b=0$ with $a,b\in\RR,a\neq 0$ and $a^2-b<0$
(see \eqref{ab3}), then the two points of $\tilde C$ which are mapped to the non-tangential points $q$ and $\ol q$ by $\Phi$ are 
\begin{align}\label{lift1}
(\zeta,\eta)=\big(
-a + i\sqrt{b-a^2},-a + i\sqrt{b-a^2}
\big)
\qandq
\big(
-a - i\sqrt{b-a^2},-a - i\sqrt{b-a^2}
\big).
\end{align}

\end{lemma}

\proof
As in the proof of Proposition \ref{p:mtl1}, we put $\tilde C_i=R^i(\tilde C)$.
Using the equation \eqref{ab3} of $\tilde C$, we readily obtain that $\tilde C\cap \tilde C_{i}= \tilde C\cap \tilde\Sigma^{(n-i)}$ if $i\neq 0$.
As in the proof of Proposition \ref{p:mtl1}, each of these intersections consists of two points which are exchanged by $\sigma$, and all these points are mapped to nodes of $C$. All $(n-1)$ nodes of $C$ are obtained this way, and 
$\Phi(\tilde C\cap \tilde C_i)= \Phi(\tilde C\cap \tilde C_{n-i})$ for any $i$ because $\tilde C\cap\tilde C_i=R^n(\tilde C)\cap \tilde C_i=R^i(\tilde C\cap \tilde C_{n-i})$.
Hence, if $n$ is odd, then the set 
$$\bigcup_{1\le i\le\frac{n-1}2}(\tilde C\cap \tilde C_i)$$
is mapped bijectively onto the set of all nodes of $C$, while if $n$ is even, then the set 
\begin{align}\label{node1}
\bigcup_{1\le i<\frac{n}2}(\tilde C\cap \tilde C_i)
\end{align}
is mapped bijectively onto the set of all non-real nodes of $C$ and 
the intersection $\tilde C\cap\tilde C_{\frac n2}$ is mapped to the unique real node of $C$ as in the proof of Proposition \ref{p:mtl1}.

We have an obvious relation
\begin{align}\label{node2}
C\cap \Sigma_0 = \bigcup_{0\le i<n} \Phi\big(\tilde C\cap\tilde\Sigma\upi\big)
\end{align}
and $\Phi\big(\tilde C\cap\tilde\Sigma\upi\big)=\Phi\big(\tilde C\cap\tilde C_{i}\big)$ if $i\neq 0$.
Therefore, from the above description of the nodes of $C$, we have, if $n$ is odd, disposing contributions from $\tilde C\cap\tilde\Sigma\upi$ with $i>\frac{n-1}2$, 
$$
C\cap \Sigma_0 = \Phi(\tilde C\cap \tilde\Sigma^{(0)})
\cup\bigcup_{1\le i\le\frac{n-1}2}\Phi(\tilde C\cap \tilde C_i)
$$
and, if $n$ is even, disposing contributions from $\tilde C\cap\tilde\Sigma\upi$ with $i>\frac{n}2$, 
\begin{align}\label{node3}
C\cap \Sigma_0 
= \Phi(\tilde C\cap \tilde\Sigma^{(0)})
\cup\Phi(\tilde C\cap \tilde C_{\frac n2})
\cup\bigcup_{1\le i<\frac{n}2}\Phi(\tilde C\cap \tilde C_i).
\end{align}
These identities show that, in either case, the non-tangential intersection is exactly $\Phi(\tilde C\cap \tilde\Sigma^{(0)})$.
Calculating the intersection $\tilde C\cap \tilde\Sigma^{(0)}$ from their equations, we obtain \eqref{lift1}.
\proofend

\medskip
Similarly, for the twisted case, we have:

\begin{lemma}\label{l:q2}
Suppose $n$ is even.
If a lift $\tilde C$ of a symmetric twisted minitwistor line $C$ is defined by 
$\zeta\eta+a(\zeta-\eta)+b=0$
with $a,b\in\RR,a\neq 0$ and $a^2+b<0$ (see \eqref{ab4}),
then the two points of $\tilde C$ which are mapped to the non-tangential points $q$ and $\ol q$ by $\Phi$ are 
\begin{align}\label{lift2}
\big(
a + i\sqrt{-b-a^2},-a - i\sqrt{-b-a^2}
\big)
\qandq
(\zeta,\eta)=\big(
a - i\sqrt{-b-a^2},-a + i\sqrt{-b-a^2}
\big).
\end{align}
\end{lemma}

\proof
Since this can be proved similarly to the previous lemma, we only give an outline.
We keep the notations from the previous lemma.
Using the equation \eqref{ab4} of $\tilde C$, we obtain that 
$\tilde C\cap \tilde C_{i} = \tilde C\cap \tilde\Sigma^{(\frac n2-i)}$ modulo $n$ if $i\neq 0$.
Disposing the case $i<\frac n2$ this time for the same reason, the set 
\begin{align*}
\bigcup_{\frac{n}2<i<n}(\tilde C\cap \tilde C_i)
\end{align*}
is mapped bijectively to the set of non-real nodes of $C$, and
$\tilde C\cap\tilde C_{\frac n2}$ is mapped to the unique real node of $C$.
The equality \eqref{node2} still holds, and it means 
\begin{align}\label{node4}
C\cap \Sigma_0 
= \Phi(\tilde C\cap \tilde\Sigma^{(\frac n2)})
\cup\bigcup_{0 \le i<\frac{n}2}\Phi(\tilde C\cap \tilde C_{i+\frac n2 }).
\end{align}
The image $\Phi(\tilde C\cap \tilde C_{\frac n2 })$ is the unique real node of $C$
and $\cup_{0 < i<\frac{n}2}\Phi(\tilde C\cap \tilde C_{i+\frac n2 })$ is the set of 
non-real nodes of $C$.
Therefore, $\Phi(\tilde C\cap \tilde\Sigma^{(\frac n2)})$ is exactly the set of non-tangential intersection points.
Solving the quadratic equation, we obtain \eqref{lift2}.
\proofend

\medskip
In the following, we denote by $p_0$ the singular point of $\Sigma_0$ lying over $z=0$ and by $p_{\infty}$ the other singular point of $\Sigma_0$.
Both of these are real.
As we mentioned, the branch divisor $\Sigma_0=\{(v,z)\set v^2=4z^n\}\subset C(\Lmd)$ of the double covering
$\Pi_0:Q/\ZZ_n\lras C(\Lmd)$ is irreducible and a rational curve with cusps at $p_0$ and $p_{\infty}$ if $n$ is odd and consists of two smooth rational curves touching at $p_0$ and $p_{\infty}$ if $n$ is even.
In both cases, the irreducible components of $\Sigma_0$ are real under the induced real structure $(v,z)\longmapsto (\ol v,\ol z)$.
Under the limit $a_1,a_2,\dots,a_n\lras 0$ and $a_{n+1},a_{n+2},\dots,a_{2n}\lras \infty$ as in \eqref{lim1}, 
the curve $\Sigma_0$ is the limit of the hyperelliptic curve $\Sigma$. 
Of course, $\Sigma$ varies as the parameters $a_1,a_2,\dots,a_{2n}$ move.
We now describe the limits of the real and pure imaginary
circles introduced in Section \ref{s:degn}.

If $n$ is even, then in this limit, half of the pure imaginary
circles collapse to $p_0$, while the remaining half collapse
to $p_{\infty}$. Exactly two real circles, namely the inner
middle circle and the outer circle, survive. Each limiting
circle maps two-to-one onto one of the two closed arcs of
$\Lmd\us$ with endpoints $0$ and $\infty$. These two circles
meet at $p_0$ and $p_{\infty}$ and together divide each of the
two irreducible components of $\Sigma_0$ into two halves.
Consequently, their union divides $\Sigma_0$ into four
domains, each homeomorphic to a half-disk.
On the other hand, if $n$ is odd, then exactly one of the pure imaginary circles (the inner middle one) and exactly one of the real circles (the outer one) survive in the limit, and these limit circles meet at $p_0$ and $p_{\infty}$, the two cusps of $\Sigma_0$.
Again, these two circles divide the irreducible branch curve $\Sigma_0$ into four domains. 
Each of these domains is also identified with a half-disk.
So in either case, the limit $\Sigma_0$ is divided into four quarters by two circles.

\begin{figure}[htbp]
\centering
\includegraphics[width=\linewidth]{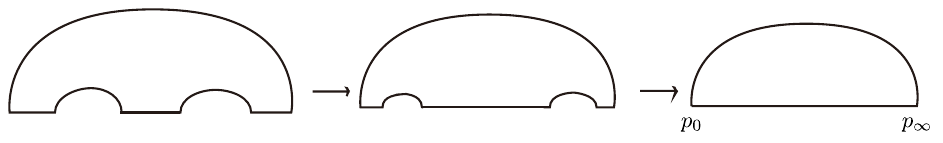}

\medskip
\includegraphics[width=\linewidth]{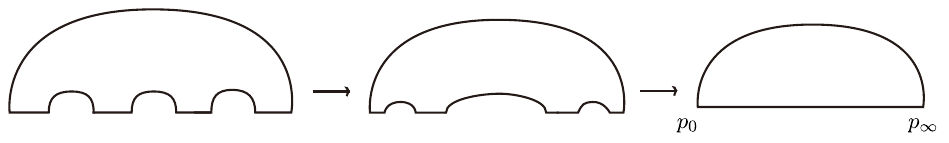}
\caption{Degeneration of the closed quarters of the branch
curves as $t\to0$ for $g=2$ and $g=3$.}
\label{f:quarter-degeneration}
\end{figure}

Let $\rho$ be the automorphism of $Q/\ZZ_n$ induced by
\[
(\zeta,\eta)\longmapsto
\bigl(e^{\frac{\pi i}n}\zeta,e^{-\frac{\pi i}n}\eta\bigr)
\]
on $Q$. Since the square of this automorphism generates the
$\ZZ_n$-action on $Q$, the induced automorphism $\rho$ is an
involution on $Q/\ZZ_n$. It also induces an involution on
$C(\Lmd)$, given in the coordinates $(z,v)$ by
\[
(z,v)\longmapsto(z,-v).
\]
In particular, it preserves $\Sigma_0$, and its restriction to
$\Sigma_0$ is the limiting hyperelliptic involution.

Now we can provide the main result of this subsection.
We set
$$\ol{\ms D}_{\rm ord}=\{(a,b)\in\RR^2\set a\ge 0, a^2\le b\}\qandq
\ol{\ms D}_{\rm tw}=\{(a,b)\in\RR^2\set a\ge 0, a^2+b\le 0\}.$$
When speaking of their natural closures below, we further adjoin one point at infinity to each of these domains.
These are the closures of the sets $\ms D_{\rm ord}$ and $\ms D_{\rm tw}$
defined in \eqref{dom1} and \eqref{dom2}, which are halves of the spaces of symmetric ordinary and twisted minitwistor lines in $Q/\ZZ_n$ respectively.
Points on a boundary component $\{(0,b)\set b\ge 0\}\subset \ol{\ms D}_{\rm ord}$ and
$\{(0,b)\set b\le 0\}\subset\ol{\ms D}_{\rm tw}$
respectively parameterize multiple minitwistor lines, and
another boundary component $\{(a,b)\set a^2-b=0\}\subset \ol{\ms D}_{\rm ord}$ and
$\{(a,b)\set a^2+b= 0\}\subset\ol{\ms D}_{\rm tw}$ respectively parameterize boundary minitwistor lines that are symmetric.

\begin{proposition}\label{p:Sigma0}
Associating with each symmetric ordinary minitwistor line
represented by \eqref{ab3} the former non-tangential point
in \eqref{lift1} induces a bijection from $\ms D_{\rm ord}$ onto one of the four quarters of \(\Sigma_0\).
This correspondence extends continuously to their natural
closures.

If \(n\) is even, associating with each symmetric twisted
minitwistor line represented by \eqref{ab4} the former
non-tangential point in \eqref{lift2} likewise induces a
bijection from $\ms D_{\rm tw}$ onto one of the
four quarters of \(\Sigma_0\). This correspondence again
extends continuously to their natural closures.
\end{proposition}

The analogous parametrizations of the remaining quarters are
obtained by applying the real structure and the involution
\(\rho\).

\proof
For the ordinary family, put
$
\xi=-a+i\sqrt{b-a^2}.
$
As \((a,b)\) varies under the conditions \(a>0\) and
\(b-a^2>0\), the point \(\xi\) varies exactly once over the
second quadrant of \(\CC\), with inverse
$
a=-\operatorname{Re}\xi,\qquad b=|\xi|^2.
$
By Lemma \ref{l:q}, the former non-tangential point is the
image of \((\xi,\xi)\). Its coordinates on \(\Sigma_0\) are
$
(z,v)=(\xi^2,2\xi^n).
$
Since \(\xi\mapsto\xi^2\) maps the second quadrant
bijectively onto the lower half-plane, the point
\((\xi^2,2\xi^n)\) sweeps out one quarter of \(\Sigma_0\)
exactly once.

If \(n\) is even, the same argument applies to the twisted
family. Namely, put
$
\xi=a+i\sqrt{-b-a^2}.
$
As \((a,b)\) varies under the conditions \(a>0\) and
\(a^2+b<0\), the point \(\xi\) varies exactly once over the
first quadrant. By Lemma \ref{l:q2}, the former
non-tangential point is the image of \((\xi,-\xi)\), whose
coordinates on \(\Sigma_0\) are
$
(z,v)=(-\xi^2,2\xi^n).
$
Since \(\xi\mapsto-\xi^2\) maps the first quadrant
bijectively onto the lower half-plane, this point likewise
sweeps out one quarter of \(\Sigma_0\) exactly once.

These formulas also give continuous extensions to the
natural closures. In the ordinary case, the boundary
components \(a=0\) and \(b=a^2\) correspond respectively to
the positive imaginary and negative real axes in the
\(\xi\)-plane. If \(n\) is odd, their images are arcs of the
limiting pure imaginary and real circles, respectively. If
\(n\) is even, their images are the two complementary arcs
of the real circle on the relevant component of \(\Sigma_0\).

In the twisted case, the boundary components \(a=0\) and
\(a^2+b=0\) correspond respectively to the positive imaginary
and positive real axes in the \(\xi\)-plane. Their images are
the two complementary arcs of the real circle on the
relevant component of \(\Sigma_0\). The points \(\xi=0\) and
\(\xi=\infty\) correspond to the two singular points of
\(\Sigma_0\).
\proofend

\section{Convergence of minitwistor lines}
In this section, we prove that the family of minitwistor lines in $\ms T$ converges to 
that of minitwistor lines in the limiting space $Q/\ZZ_n$, using the results in the previous section.

\subsection{The relative Picard space and the squaring map}\label{ss:relP}
Take any minitwistor space $\ms T$ as in \eqref{tw2} and
let $\Delta\subset\CC$ be an open disk containing $0$ and $1$.
To fix a degeneration of $\ms T$ to $Q/\ZZ_n$ parameterized by $\DDD$,
for each index $i$ with $1\le i\le 2n$,
we take any holomorphic path
$$
\big\{a_i(t)\in\PP^1\set t\in\DDD\big\},\quad\text{ with $a_i(1)=a_i$},
$$
such that $a_1(t)<\cdots<a_{2n}(t)$
for $t\in I\setminus\{0\}$ where $I:=\DDD\cap\RR$
and
$$
a_i(0)=
\begin{cases}
0 & \text{if } 1\le i\le n,\\
\infty & \text{if } n<i\le 2n.
\end{cases}
$$
For $t\in\DDD\setminus\{0\}$, let
$$
\Sigma_t:=\big\{v^2=(z-a_1(t))(z-a_2(t))\dots(z-a_{2n}(t))\big\}
\subset C(\Lmd).
$$
At $t=0$, let $\Sigma_0$ denote the rescaled limit described in Section~\ref{s:degn}. Let
$\bm\Sigma\lras\DDD$ be the resulting flat family of compact
curves, whose fiber over $t\in\DDD$ is $\Sigma_t$.
Let $\ms T_t$ be the double covering of the cone $C(\Lmd)$ branched along $\Sigma_t$.
Then $\ms T_0=Q/\ZZ_n$ and $\ms T_1=\ms T$.
From Section \ref{s:degn}, the family $\{\ms T_t\set t\in\DDD\}$ is a degeneration of $\ms T$ to $Q/\ZZ_n$.

For each $k\in\ZZ$, let
\begin{equation}\label{rPic}
\Pic^k(\bm\Sigma/\DDD)\lras\DDD
\end{equation}
be the relative Picard space parametrizing line bundles of
total degree $k$ on the fibers of $\bm\Sigma\lras\DDD$.
It is represented by an algebraic space
\cite[Tag 0D2C]{Stacks}. We will mainly use the cases
$k=g$ and $k=2g$, where $g=n-1$ as before. The fiber of
\eqref{rPic} over $t\in\DDD$ is naturally identified with
$\Pic^k(\Sigma_t)$.

If $n$ is even, the central fiber $\Sigma_0$ is reducible,
and the total degree does not determine a component of its
Picard space. For each pair $\underline d=(d_1,d_2)$ of
integers, we write $\Pic^{\underline d}(\Sigma_0)$ for the
locus of line bundles of multidegree $\underline d$. Thus
$$
\Pic^k(\Sigma_0)
=
\coprod_{d_1+d_2=k}
\Pic^{(d_1,d_2)}(\Sigma_0).
$$
Since $\bm\Sigma\lras\DDD$ is a proper flat family of reduced
and connected curves, the projection \eqref{rPic} is smooth
\cite[Tag 0DPK]{Stacks}. Hence, since $\DDD$ is non-singular, the
total space $\Pic^k(\bm\Sigma/\DDD)$ is also non-singular.

We denote by
\begin{equation}\label{mfT}
\mf T:\Pic^g(\bm\Sigma/\DDD)
\lras\Pic^{2g}(\bm\Sigma/\DDD)
\end{equation}
the relative squaring map, given by $L\mapsto L^{\otimes2}$.
For each $t\in\DDD$, under the natural identifications of
the fibers with $\Pic^g(\Sigma_t)$ and
$\Pic^{2g}(\Sigma_t)$, respectively, the restriction of
$\mf T$ to the fiber over $t$ is identified with the
squaring map
\[
\mf t_t:\Pic^g(\Sigma_t)\lras\Pic^{2g}(\Sigma_t).
\]

Let $J(\Sigma_0)$ denote the generalized Jacobian of
$\Sigma_0$, namely the identity component of its Picard
space.

\begin{proposition}\label{p:squaring}
The generalized Jacobian $J(\Sigma_0)$, namely the identity
component of the Picard space of $\Sigma_0$, is given by
\[
J(\Sigma_0)\simeq
\begin{cases}
\CC^g, & n\ \text{odd},\\
\CC^*\times\CC^{g-1}, & n\ \text{even}.
\end{cases}
\]Consequently, if $n$ is odd, the squaring map
\[
\mf t_0:\Pic^g(\Sigma_0)\lra\Pic^{2g}(\Sigma_0)
\]
is an isomorphism. If $n$ is even, then, on every fixed
multidegree component, the squaring map
\[
\mf t_0:\Pic^{\underline d}(\Sigma_0)
\lra\Pic^{2\underline d}(\Sigma_0)
\]
is a two-to-one unramified covering.

For $t\ne0$, the map
\[
\mf t_t:\Pic^g(\Sigma_t)\lra\Pic^{2g}(\Sigma_t)
\]
is an unramified covering of degree $2^{2g}$. Moreover, the
relative squaring map $\mf T$ is a local biholomorphism.
\end{proposition}

\proof
Suppose first that $n$ is odd. Then $\Sigma_0$ has two cusps, $p_0$ and $p_{\infty}$, whose local rings are isomorphic to
$
\CC[[\lambda^2,\lambda^n]]\subset\CC[[\lambda]].
$
By the normalization exact sequence for the Picard group
\cite[Chapter II, Exercise 6.9]{Hartshorne}
(see also \cite{Ros54} for generalized Jacobians), and since
the normalization of $\Sigma_0$ is rational, we have
\[
\Pic^0(\Sigma_0)
\simeq
\bigoplus_{p\in\Sing(\Sigma_0)}
\widetilde{\ms O}_p^*/\ms O_p^*.
\]
For each cusp, the quotient on the right is a vector group of dimension $(n-1)/2$, corresponding to the gaps $1,3,\ldots,n-2$ of the semigroup $\langle2,n\rangle$. Hence
\[
J(\Sigma_0)\simeq\CC^g.
\]

Next, suppose that $n$ is even, and write $n=2m$. Then $\Sigma_0$ consists of two smooth rational curves meeting at two $A_{2m-1}$-singularities. At either singular point $p$, we have
\[
\ms O_p\simeq
\left\{
(f_1,f_2)\in\CC[[\lambda]]\oplus\CC[[\lambda]]
\mathrel{}\middle|\mathrel{}
f_1\equiv f_2\pmod{\lambda^m}
\right\},
\]
and therefore
\[
Q_p:=\widetilde{\ms O}_p^*/\ms O_p^*
\simeq
\bigl(\CC[[\lambda]]/(\lambda^m)\bigr)^*
\simeq\CC^*\times\CC^{m-1}.
\]
Since the normalization is the disjoint union of two rational curves, the normalization exact sequence gives
\[
1\longrightarrow\CC^*
\longrightarrow(\CC^*)^2
\longrightarrow Q_{p_0}\times Q_{p_{\infty}}
\longrightarrow J(\Sigma_0)
\longrightarrow1.
\]
The quotient $(\CC^*)^2/\CC^*$ represents the relative rescaling of the two components and eliminates one of the two $\CC^*$-factors. Thus
\[
J(\Sigma_0)=\Pic^{(0,0)}(\Sigma_0)
\simeq\CC^*\times\CC^{2m-2}
=\CC^*\times\CC^{g-1}.
\]

Choose a line bundle $L$ in a fixed degree or multidegree component and use $L$ and $L^{\otimes2}$ as origins of the source and target Picard components. Under these identifications, the squaring map is multiplication by $2$ on $J(\Sigma_0)$. Hence it is an isomorphism when $n$ is odd. When $n$ is even, it is given by
\[
(z,w)\longmapsto(z^2,2w)
\]
on $\CC^*\times\CC^{g-1}$, and is therefore a two-to-one unramified covering.

For $t\ne0$, the variety $\Pic^0(\Sigma_t)$ is a complex torus of dimension $g$. After choosing origins, $\mf t_t$ is multiplication by $2$, whose kernel is
\[
\Pic^0(\Sigma_t)[2]\simeq(\ZZ/2\ZZ)^{2g}.
\]
Thus $\mf t_t$ is an unramified covering of degree $2^{2g}$.

Finally, $\mf T$ lies over the identity of $\DDD$. Its differential induces the identity on the tangent space of the base and multiplication by $2$ on each relative tangent space. It is therefore an isomorphism at every point. Since the relative Picard spaces are non-singular, the inverse function theorem shows that $\mf T$ is a local biholomorphism.
\proofend

\medskip
In the next subsection, we use Proposition \ref{p:squaring}
to study, in families, the square roots of the line-bundle
classes parametrized by the quarters of the branch curves.  


\subsection{The Seifert surfaces and their convergence}\label{ss:Seifert}
By Proposition \ref{p:Sigma0}, $\ms D_{\rm ord}$
(see \eqref{dom1}) parametrizes the symmetric ordinary
minitwistor lines. Assigning to each such line one of its two
non-tangential points, as specified in the proposition,
identifies $\ms D_{\rm ord}$ with the interior of one of the
four quarters of $\Sigma_0$. When $n$ is even, the same
statements hold for the symmetric twisted minitwistor lines
and $\ms D_{\rm tw}$ (see \eqref{dom2}). We fix one of the four closed quarters and denote it by
$\Sigma_0''$; there is no canonical choice.

The equations defining $\bm\Sigma\lras\DDD$ are invariant
under $(v,z,t)\mapsto(\ol v,\ol z,\ol t)$ and
$(v,z,t)\mapsto(-v,z,t)$. Hence, for $t\in I:=\DDD\cap\RR$, their
restrictions to $\Sigma_t$ give the real structure and the
hyperelliptic involution, respectively; in
particular, these formulas remain valid on the central fiber.
The preceding description of the limiting real and pure
imaginary circles shows that the four closed quarters of
$\Sigma_0$ are the limits of the four closed fundamental
domains in $\Sigma_t$. Accordingly, for
$t\in I\setminus\{0\}$, we choose the corresponding closed
quarters $\Sigma_t''$ so that they converge to $\Sigma_0''$
as $t\to0$.
The singularities $p_0$ and $p_{\infty}$ of $\Sigma_0$ are exactly the two corners of 
$\Sigma''_0$.

In \cite{H26}, we introduced a Seifert surface in the real
locus of the Jacobian of the hyperelliptic branch curve in
order to identify the global structure of the associated
EW space arising from a gravitational instanton
of type $A_{2n-1}$. This surface may be equivalently defined
in a suitable Picard space. We use this Picard-theoretic
formulation, since it extends naturally to the limiting branch
curve $\Sigma_0$, even though $\Sigma_0$ is singular and can be reducible.

We retain the notation of the previous subsection. Let
$\pi_t:\Sigma_t\lras\Lmd$ denote the double covering. For
$t\neq0$, let $r_{i,t}$ be the ramification point of $\pi_t$
lying over the branch point $a_i(t)$. We extend this notation by setting
$r_{1,0}:=p_0$.

Using $I=\DDD\cap\RR$, write
\[
\bm\Sigma_I:=\bm\Sigma|_I,\qquad
\bm\Sigma_I'':=\bigcup_{t\in I}\Sigma_t''.
\]
Let $h_t=\ms O_{\PP^{n+1}}(1)|_{\Sigma_t}$ be the hyperplane
class. The relative Abel map gives the map 
\[
\alpha:\bm\Sigma_I\longrightarrow
\Pic^{2g}(\bm\Sigma/\DDD)|_I,\qquad
q\in\Sigma_t\longmapsto[h_t-q-\ol q].
\]
At either singular point $p=p_0,p_\infty$ of $\Sigma_0$, the
divisor $p+\ol p$ is interpreted as the Cartier divisor because
$p+\ol p=2p=\pi_0^*(\pi_0(p))$.

Let $\alpha_t:\Sigma_t\lras\Pic^{2g}(\Sigma_t)$ denote the
restriction of $\alpha$ to $\Sigma_t$, and define
\[
\mf V_t:=\alpha_t(\Sigma_t''),\qquad
\mf V:=\alpha(\bm\Sigma_I'')
      =\bigcup_{t\in I}\mf V_t.
\]
Suppose that $t\neq0$. If $q$ belongs to a pure
imaginary circle, then $\ol q$ is the image of $q$ under the
hyperelliptic involution. Hence
$q+\ol q\sim\pi_t^*\ms O_{\Lmd}(1)$, and therefore
$
\alpha_t(q)=[\pi_t^*\ms O_{\Lmd}(g)].
$
Thus $\alpha_t|_{\Sigma_t''}$ contracts all the pure imaginary
semicircles in the boundary of $\Sigma_t''$ to a single point.

Assume in addition that $g>1$, and suppose that
$q_1,q_2\in\Sigma_t''$ have the same image under $\alpha_t$.
Then
$q_1+\ol q_1\sim q_2+\ol q_2$. If these two effective divisors
are distinct, they belong to a positive-dimensional linear
system of degree two. By \cite[Proposition 2.1]{H25}, this
linear system is the hyperelliptic pencil. Hence both $q_1$
and $q_2$ belong to pure imaginary circles. It follows that
$\alpha_t|_{\Sigma_t''}$ is injective away from the pure
imaginary semicircles. Consequently, $\mf V_t$ is obtained
from $\Sigma_t''$ by contracting these semicircles to one
point.

When $g=1$, the restriction
$\alpha_t|_{\Sigma_t''}$ is not injective. We therefore
retain $\Sigma_t''$ as the parameter space and regard
$\mf V_t$ as the image of the parametrized map
$\alpha_t|_{\Sigma_t''}$. None of the arguments below
requires this map to be injective.

For $t\neq0$, after translating $\Pic^{2g}(\Sigma_t)$ to
$\Pic^0(\Sigma_t)$ by $x\mapsto[x-2gr_{1,t}]$ and identifying
$\Pic^0(\Sigma_t)$ with the Jacobian ${\rm J}(\Sigma_t)$, the restriction
$\alpha_t|_{\Sigma_t''}$ agrees with the map used in
\cite[Section 6.2]{H26}.

For each $t\in I\setminus\{0\}$, the boundary
construction in \cite[Section 6.2]{H26} determines a lift
$
\tilde\alpha_t
$
of $\alpha_t$; namely, a map $\tilde\alpha_t:\Sigma_t''\longrightarrow\Pic^g(\Sigma_t)$ that satisfies
$\mf t_t\circ\tilde\alpha_t
=
\alpha_t|_{\Sigma_t''}.$
We then put
$$\tilde{\mf V}_t:=\tilde\alpha_t(\Sigma_t'')$$
and call it a lift of $\mf V_t$.
At this stage, this is not defined for $t=0$.

\begin{proposition}\label{p:limit-lift}
The maps $\tilde\alpha_t$, $t\in I\setminus\{0\}$, extend
continuously across $t=0$ to a lift
$
\tilde\alpha_0:\Sigma_0''\longrightarrow\Pic^g(\Sigma_0)
$
of $\alpha_0|_{\Sigma_0''}$ under $\mf t_0$. Consequently,
if we define $\tilde{\mf V}_0:=\tilde\alpha_0(\Sigma_0'')$, then as $t\to 0$,
\[
\tilde{\mf V}_t\longrightarrow\tilde{\mf V}_0.
\]
When $n$ is odd, $\tilde\alpha_0$ is unique. When $n$ is
even, $\alpha_0|_{\Sigma_0''}$ has exactly two lifts under
$\mf t_0$, and $\tilde\alpha_0$ is one of them.
\end{proposition}

\proof
First, suppose that $n$ is odd. Since $r_{1,t}\to p_0$ as $t\to0$, the map $t\mapsto r_{1,t}$ defines a continuous section of $\bm\Sigma_I\to I$.
For each $t\in I$, put
$e_t:=[2gr_{1,t}]=\alpha_t(r_{1,t})\in\mf V_t$
and use this as a reference point. 
For $t\neq0$, the
boundary condition in \cite[(6.6)]{H26} gives
$
\tilde\alpha_t(r_{1,t})=[gr_{1,t}].
$

Since $g=n-1$ is even, we can write $gr_{1,t}=(g/2)(2r_{1,t})$. The divisors $2r_{1,t}$ form a continuous family of Cartier divisors and converge to the Cartier divisor $2p_0$ as $t\to0$. Consequently,
$$
[gr_{1,t}]\longrightarrow[gp_0],\qquad \mf t_0([gp_0])=[2gp_0]=e_0.
$$
As $n$ is odd, by Proposition \ref{p:squaring}, the squaring map $\mf t_0$ is an isomorphism. 
So we may define
\[
\tilde\alpha_0
:=
\mf t_0^{-1}\circ\alpha_0|_{\Sigma_0''},
\qquad
\tilde{\mf V}_0:=\tilde\alpha_0(\Sigma_0'').
\]
Since the relative squaring map $\mf T$ is a local
biholomorphism, for every $q_0\in\Sigma_0''$, its local
inverse at $\tilde\alpha_0(q_0)$ gives a unique local lift of
$\alpha$ extending $\tilde\alpha_0$ near $q_0$.
These local
continuations agree on overlaps by local uniqueness. Since
$\Sigma_0''$ is compact, they give, for all sufficiently
small $t$, a lift
\[
\widehat\alpha_t:\Sigma_t''\longrightarrow\Pic^g(\Sigma_t)
\]
of $\alpha_t|_{\Sigma_t''}$ extending $\tilde\alpha_0$.


To show $\tilde{\alpha}_t=\widehat{\alpha}_t$,
for all sufficiently small $t\neq0$, both arise as lifts of $\alpha_t:\Sigma_t''\to\mf V_t$ taking the reference point $r_{1,t}$ to $[gr_{1,t}]$.
The convergence $[gr_{1,t}]\to[gp_0]$ and the local
uniqueness of the inverse of $\mf T$ imply that
$
\widehat\alpha_t(r_{1,t})=[gr_{1,t}]
=\tilde\alpha_t(r_{1,t})
$
for all sufficiently small $t\neq0$. Thus
$\widehat\alpha_t$ and $\tilde\alpha_t$ are two lifts of
$\alpha_t|_{\Sigma_t''}$ agreeing at $r_{1,t}$. Since
$\Sigma_t''$ is connected, the uniqueness of lifting gives
$\widehat\alpha_t=\tilde\alpha_t$. This proves the assertion
when $n$ is odd.

Suppose next that $n$ is even. Since $g=n-1$ is odd, the
preceding argument based on
$gr_{1,t}=(g/2)(2r_{1,t})$ cannot be applied. We instead
choose a continuous family of reference points
$q_t\in\partial\Sigma_t''$, $t\in I$, such that each $q_t$
lies in the interior of the middle real semicircle. For
$t\neq0$, $q_t$ is therefore distinct from the endpoints
$r_{1,t}$ and $r_{2n,t}$ of this semicircle.
Let $q_t'$ denote the image of $q_t$ under the hyperelliptic
involution. Since the middle real semicircle survives in the
limit and $q_0$ lies in its interior, $q_0$ is a smooth point
of $\Sigma_0$. Moreover, $q_t'\to q_0'$, and $q_0'$ is also a smooth point
of $\Sigma_0$.

Since $q_t$ is real and $h_t\sim n(q_t+q_t')$, we have $\alpha_t(q_t)=[h_t-2q_t]=[(n-2)q_t+nq_t']$. We therefore put
$$
e_t:=[(n-2)q_t+nq_t'],\qquad b_t:=\left[\frac{n-2}{2}q_t+\frac n2q_t'\right],
$$
so that $\mf t_t(b_t)=e_t$. 
The point $q_t$ lies in the interior of the middle real
semicircle, and the boundary condition in \cite{H26} gives
$\tilde\alpha_t(q_t)=b_t$.

Because $q_0$ and $q_0'$ are smooth points of $\Sigma_0$, the divisors defining $b_t$ form a continuous family, and hence $b_t\to b_0:=\left[\frac{n-2}{2}q_0+\frac n2q_0'\right]$ as $t\to0$. 
Since $\Sigma_0''$ is simply connected and $\mf t_0$ is a
two-to-one covering by Proposition \ref{p:squaring}, the map
$\alpha_0|_{\Sigma_0''}$ has exactly two lifts. We denote by
$\tilde\alpha_0$ the lift satisfying
$\tilde\alpha_0(q_0)=b_0$, and put
$\tilde{\mf V}_0:=\tilde\alpha_0(\Sigma_0'')$.

As in the odd case, the local biholomorphism property of
$\mf T$ gives a continuation
$\widehat\alpha_t$ of $\tilde\alpha_0$. Since $b_t\to b_0$,
local uniqueness gives
$\widehat\alpha_t(q_t)=b_t=\tilde\alpha_t(q_t)$
for all sufficiently small $t\neq0$. Both maps are lifts of
$\alpha_t|_{\Sigma_t''}$, and hence the uniqueness of lifting
on the connected space $\Sigma_t''$ gives
$\widehat\alpha_t=\tilde\alpha_t$. Consequently,
$\tilde{\mf V}_t\to\tilde{\mf V}_0$, completing the proof.
\proofend

\medskip
When $n$ is even, Propositions \ref{p:Sigma0} and
\ref{p:br2} show that the tangency divisors of the ordinary
and twisted symmetric minitwistor families define the two
lifts of $\mf V_0$
under $\mf t_0$. Since $\mf t_0$ is a two-to-one covering by
Proposition \ref{p:squaring}, these exhaust all possible
lifts. Consequently, the limiting lift
$\widetilde{\mf V}_0$ obtained above
corresponds to either the ordinary or the twisted symmetric minitwistor family.
We do not need to determine which one.

\subsection{Proof of the convergence theorem}
Let $\ms T_t$, $t\in\DDD$, be the double covering of $C(\Lmd)$ with branch $\Sigma_t$.
The family $\{\ms T_t\set t\in I\}$, $I=\DDD\cap \RR$ is a one-parameter degeneration of minitwistor spaces, with $\ms T_0=Q/\ZZ_n$.
We want to prove that minitwistor lines in $\ms T_t$ converge to those in $Q/\ZZ_n$.
Let $\{\Sigma_t''\}_{t\in I}$ with $I=\DDD\cap\RR$ be the
corresponding family of quarters of the branch curves.
Take any continuous section $\{q_t\in\Sigma_t''\}_{t\in I}$ of $\bm\Sigma''_I
\lras I$
through a point $q_0\in\Sigma_0''$.
For each $t\in I$, let $C_t\subset\ms T_t$ be the member of
the natural closure of the symmetric family corresponding to
$q_t$ under the parametrization determined by the lift map
$\tilde\alpha_t$. We retain $q_t$ as part of the
parametrization, since a singular boundary point of
$\tilde{\mf V}_t$ need not determine a unique member.

\begin{proposition}\label{p:conv}
As $t\in\DDD\cap\RR$ approaches $0$, we have
$
C_t\longrightarrow C_0
$
as projective $1$-cycles in $\PP^{n+2}$.
\end{proposition}

\proof
We first suppose that $q_0$ belongs to the interior of
$\Sigma_0''$.
Let
$\widetilde\alpha_t:\Sigma_t''\lra
\widetilde{\mf V}_t\subset\Pic^g(\Sigma_t)$
denote the lift map, and put
$\ell_t:=\widetilde\alpha_t(q_t)$.
Proposition \ref{p:limit-lift} and the convergence
$q_t\to q_0$ imply that
$
\ell_t\longrightarrow
\ell_0=\widetilde\alpha_0(q_0).
$
For $t\neq0$, let
$s_t\in H^0(\PP^{n+1},\ms O_{\PP^{n+1}}(1))$
be a nonzero linear form defining the hyperplane associated
with $C_t$, and let $D_t$ be its tangency divisor, with
multiplicities allowed. Since $C_t$ is the member selected by
the lift $\widetilde\alpha_t$ at $q_t$, we have
\[
\operatorname{div}(s_t|_{\Sigma_t})
=q_t+\ol q_t+2D_t,
\qquad
\ell_t=[D_t].
\]

By the discussion following Proposition
\ref{p:limit-lift}, the chosen lift
$\widetilde{\mf V}_0$ determines the relevant family described
in the previous section. Proposition \ref{p:Sigma0}
parameterizes its natural closure by $\Sigma_0''$.
Let $C_0\subset Q/\ZZ_n$ be the curve corresponding to
$q_0$ under this parametrization.
Let
$s_0\in H^0(\PP^{n+1},\ms O_{\PP^{n+1}}(1))$
be a nonzero linear form defining the corresponding
hyperplane, and let $D_0$ be its tangency divisor. Then
\[
\ell_0=[D_0],
\qquad
\operatorname{div}(s_0|_{\Sigma_0})
=q_0+\ol q_0+2D_0.
\]
Here we retain the parameter $q_0$, since the point
$\ell_0\in\widetilde{\mf V}_0$ alone need not determine
$C_0$ uniquely.

Locally near $\ell_0$ in the relative Picard space, choose a
family of line bundles representing its points, and let $L_t$
denote the line bundle corresponding to $\ell_t$. Since
$\bm\Sigma\lras\DDD$ is flat and its general fiber has genus
$g$, we have $p_a(\Sigma_t)=g$ for every $t$. Moreover,
$\deg L_t=g$, and hence Riemann--Roch gives
$\chi(\Sigma_t,L_t)=1$.
Since $\Sigma_0$ is a quadric section of $C(\Lmd)$,
adjunction gives
$\omega_{\Sigma_0}\simeq\pi_0^*\ms O_{\Lmd}(n-2)$.
Moreover, since $h_0$ does not pass through the vertex of the
cone, it cannot contain both points of any fiber of $\pi_0$;
hence the $g=n-1$ points of $D_0$ have distinct images on
$\Lmd$. A section of
$\omega_{\Sigma_0}\otimes L_0^{-1}$ would therefore give a
section of $\ms O_{\Lmd}(n-2)$ vanishing at $n-1$ distinct
points, and must be zero. Thus
\[
H^0(\Sigma_0,\omega_{\Sigma_0}\otimes L_0^{-1})=0.
\]
Therefore $H^1(\Sigma_0,L_0)=0$ by duality.
Upper semicontinuity then gives, after shrinking $\DDD$ if
necessary,
\[
H^1(\Sigma_t,L_t)=0,
\qquad
h^0(\Sigma_t,L_t)=1
\]
for every sufficiently small $t$. By cohomology and base change, the unique effective divisors
in the linear systems $|L_t|$ vary continuously.
Consequently, $D_t\to D_0$.

Together with $q_t\to q_0$ and the continuity of the real
structure, the relation above implies
\[
\operatorname{div}(s_t|_{\Sigma_t})
\longrightarrow
q_0+\ol q_0+2D_0
=
\operatorname{div}(s_0|_{\Sigma_0}).
\]
Since each $\Sigma_t$ is non-degenerate and a hyperplane in
$\PP^{n+1}$ is uniquely determined by its restriction to
$\Sigma_t$, it follows that
$
Z(s_t)\longrightarrow Z(s_0).
$
The double covers
$\Pi_t:\ms T_t\lra C(\Lmd)$ vary holomorphically with $t$.
Consequently,
\[
C_t
=
\Pi_t^{-1}\bigl(Z(s_t)\cap C(\Lmd)\bigr)
\longrightarrow
\Pi_0^{-1}\bigl(Z(s_0)\cap C(\Lmd)\bigr)
=
C_0
\]
as projective $1$-cycles.

Finally, suppose that $q_0$ belongs to the boundary of
$\Sigma_0''$. The parametrizations on the smooth fibers, as
well as the parametrization in Proposition \ref{p:Sigma0},
extend continuously to their natural closures, with multiple
and boundary members regarded as projective $1$-cycles.
The assertion therefore follows by taking limits from the
interior.
\proofend

\medskip
Combining Proposition \ref{p:conv} with the $S^1$-action, we obtain the convergence of the whole family of minitwistor lines.

\begin{theorem}\label{thm:conv}
Under the degeneration \eqref{lim1}, the family of minitwistor lines
on $\ms T_t$ converges to the family of minitwistor lines on
$Q/\ZZ_n$ corresponding to the hyperbolic orbifold $\BB/\ZZ_n$.
\end{theorem}

\proof
Let $\{C_t\}_{t\in I}$ be any one-parameter family of
symmetric minitwistor lines considered in Proposition
\ref{p:conv}, and let $\lambda_t\in S^1$ vary continuously
with $t$. Proposition \ref{p:conv} and the compatibility of
the $S^1$-actions with the degeneration give
\[
\lambda_t\cdot C_t\longrightarrow\lambda_0\cdot C_0.
\]

The symmetric families form slices for the $S^1$-action.
For $t\neq0$, the symmetric family is constructed in
\cite[Section 6.2]{H26}, and its $S^1$-orbits are shown to
give the complete family in
\cite[Section 6.3, especially Theorem 6.10]{H26}.
For $t=0$, the corresponding statement follows from
Propositions \ref{p:sym} and \ref{p:Sigma0}.
Hence every one-parameter family of minitwistor lines is
locally obtained in this way. This proves the theorem.
\proofend

\medskip
\noindent\textbf{AI tool disclosure.}
The author used ChatGPT 5.6 Sol to assist with
checking calculations and mathematical arguments.


\begin{thebibliography}{99}

%







%
%
%



\bibitem{GH78}
G.\,W.\,Gibbons and S.\,W.\,Hawking,
{\em Gravitational multi-instantons},
Phys. Lett. B {\bf 78} (1978), no.\,4, 430--432.

\bibitem{Hartshorne}R. Hartshorne,
{\em Algebraic Geometry},
Graduate Texts in Mathematics, vol. 52,
Springer-Verlag, New York--Heidelberg, 1977.

\bibitem{Hi78}N. Hitchin, 
{\em Polygons and gravitons}, Math. Proc. Camb. Phil. Soc. 
{\bf 85} (1979), 465-476.

\bibitem{Hi82}N. Hitchin, 
 {\em  Complex manifolds and Einstein's equations},
Lecture Notes in Math.  {\bf 970} (1982)
73--99. 


\bibitem{Hi25}N. Hitchin, 
{\em ALE spaces and nodal curves},  
Quarterly J. Math. {\bf 76} (2025)  337--347.

\bibitem{H25}N. Honda,
{\em Hyperelliptic curves, minitwistors, and spacelike Zoll spaces},
to appear in Duke Math. J.;
arXiv:2502.11388.






%





\bibitem{HN11}N. Honda, F. Nakata,
{\em Minitwistor spaces, Severi varieties, and Einstein--Weyl structure},
Ann. Global Anal. Geom. {\bf 39} (2011) 
293--323. 

\bibitem{HN22}N. Honda, F. Nakata,
{\em The Einstein--Weyl spaces associated to Segre quartic surfaces},
to appear in Algebraic Geometry and Physics.
arXiv:2208.13567. 

\bibitem{H26}N. Honda,
{\em On the twistor spaces of ALE gravitational instantons
of type $A_{\rm odd}$}.
arXiv:2603.14720.


%





\bibitem{JT85}
P.E. Jones, K.P. Tod,
{\em Minitwistor spaces and Einstein--Weyl geometry},
Class. Quan. Grav. {\bf 2} (1985), 565--577.




\bibitem{Kro89}
P.\,B.\,Kronheimer,
{\em The construction of ALE spaces as hyper-K\"ahler quotients},
J. Differential Geom. {\bf 29} (1989), no.\,3, 665--683.








\bibitem{PT93}
H.\,Pedersen, K.\,P.\,Tod,
{\em Three Dimensional Einstein-Weyl Geometry},
Adv.\,Math {\bf 97} (1993) 74--109.



 

\bibitem{Ros54}
M.\,Rosenlicht,
{\em Generalized Jacobian varieties},
Ann. of Math. (2) {\bf 59} (1954), no.\,3, 505--530.
 
\bibitem{Stacks} The {Stacks Project Authors},
{\textit{Stacks Project}},
{{https://stacks.math.columbia.edu}} {2018},

 

 



 

\end{thebibliography}
\end{document}